\documentclass[11pt, a4paper]{article}

\usepackage[whole]{bxcjkjatype}

\usepackage{bm, amsmath, amsthm, amssymb, accents, comment}
\usepackage{ascmac}
\RequirePackage{amsthm,amsmath,amsfonts,amssymb}
\RequirePackage[numbers]{natbib}
\RequirePackage{graphicx}\usepackage{tikz}
\usepackage{pgfplots}
\usepackage{subcaption}
\pgfplotsset{
	compat=newest, 
	cycle list name=exotic }

\usepackage{graphicx}
\usepackage{xcolor}\usepackage{multirow}\usepackage{amsmath,amssymb,amsfonts}\usepackage{amsthm}\usepackage{mathrsfs}\usepackage[title]{appendix}\usepackage{textcomp}\usepackage{manyfoot}\usepackage{booktabs}\usepackage{algorithm}\usepackage{algorithmicx}\usepackage{algpseudocode}\usepackage{listings}\usepackage{url}
\usepackage{mathrsfs}
\usepackage{physics}
\usepackage{bm}
\usepackage{float}
\usepackage{hyperref}
\usepackage[numbers]{natbib}

\newtheorem{theorem}{Theorem}
\newtheorem{lemma}{Lemma}
\newtheorem{corollary}{Corollary}
\newtheorem{proposition}{Proposition}

\def\t{\theta}

\usepackage{typearea}
\typearea{12}

\begin{document}
	\title{On the attainment of the Wasserstein--Cramer--Rao lower bound}

\author{Hayato Nishimori\thanks{Department of Mathematical Informatics, The University of Tokyo, e-mail: \texttt{benzene-ring-78@g.ecc.u-tokyo.ac.jp}}, \ 
	Takeru Matsuda\thanks{Department of Mathematical Informatics, The University of Tokyo \& Statistical Mathematics Unit, RIKEN Center for Brain Science, e-mail: \texttt{matsuda@mist.i.u-tokyo.ac.jp}}}

\date{}

\maketitle

	\begin{abstract}
		Recently, a Wasserstein analogue of the Cramer--Rao inequality has been developed using the Wasserstein information matrix (Otto metric).
		This inequality provides a lower bound on the Wasserstein variance of an estimator, which quantifies its robustness against additive noise.
		In this study, we investigate conditions for an estimator to attain the Wasserstein--Cramer--Rao lower bound (asymptotically), which we call the (asymptotic) Wasserstein efficiency. 
		We show a condition under which Wasserstein efficient estimators exist for one-parameter statistical models.
		This condition corresponds to a recently proposed Wasserstein analogue of one-parameter exponential families (e-geodesics). 
		We also show that the Wasserstein estimator, a Wasserstein analogue of the maximum likelihood estimator based on the Wasserstein score function, is asymptotically Wasserstein efficient in location-scale families.
	\end{abstract}
	
	\section{Introduction}\label{introduction}
The Cramer--Rao inequality is a well-known classical theorem in statistics. 
It provides a lower bound on the variance of (unbiased) estimators through the inverse of the Fisher information matrix.
An estimator is said to be (asymptotically) Fisher efficient if it attains the Cramer--Rao lower bound (asymptotically).
For one-parameter statistical models, an estimator is Fisher efficient if and only if the model is an exponential family and it is the maximum likelihood estimator (MLE) of its expectation parameter \citep{Fisher_efficiency2, Fisher_efficiency}. For general models, the MLE is asymptotically Fisher efficient under regularity conditions \citep{vV}. In information geometry, the Fisher information is adopted as a Riemannian metric on the parameter space and is closely connected to the Kullback--Leibler divergence \citep{amari2000methods}. 

The Wasserstein distance is defined as the optimal transport cost between probability distributions and it induces another geometric structure on the space of probability distributions \citep{villani2}. 
The Wasserstein geometry has been widely applied in many fields, including statistics and machine learning \citep{Chewi,PC2019,Santambrogio}. 
Recently, \citet{WIM} developed Wasserstein counterparts of information geometric concepts such as the Wasserstein information matrix and Wasserstein score function.
They also derived the Wasserstein--Cramer--Rao inequality, which gives a lower bound on the Wasserstein variance of an estimator by the inverse of the Wasserstein information matrix. 
Whereas the usual variance quantifies the accuracy of an estimator, the Wasserstein variance can be interpreted as the robustness of an estimator against additive noise \citep{affine}.
\citet{WIM} also proposed an estimator called the Wasserstein estimator as the zero point of the Wasserstein score function.

In this study, we investigate conditions for an estimator to attain the Wasserstein--Cramer--Rao lower bound (asymptotically), which we call the (asymptotic) Wasserstein efficiency. 
In Section \ref{review}, we briefly review the Wasserstein--Cramer--Rao inequality. 
In Section \ref{one-parameter}, we focus on one-parameter models and derive a condition for Wasserstein efficiency in finite samples, which corresponds to recently proposed Wasserstein analogue of one-parameter exponential families (e-geodesics) \cite{ay2024information}. 
In Section \ref{location-scale}, we focus on location-scale families and show that the Wasserstein estimator is asymptotically Wasserstein efficient.

\section{Wasserstein--Cramer--Rao Inequality}\label{review}
In this section, we briefly review the Wasserstein information matrix and Wasserstein--Cramer--Rao inequality introduced by \citet{WIM}. 
We consider a parametric density $p(x;\theta)$ on $\mathbb{R}^d$ with parameter $\theta \in \mathbb{R}^p$ in the following.

The Wasserstein score function $\Phi_i(x;\theta)$ for $i=1,\dots,p$ is defined as the solution to the partial differential equation
\begin{align}\label{WscorePDE}
	\frac{\partial}{\partial \theta_i} p(x;\theta) + \nabla_x \cdot \left( p(x;\theta) \nabla_x \Phi_i(x;\theta) \right) = 0
\end{align}
satisfying ${\rm E}_{\theta}[\Phi_i(x;\theta)] = 0$, where $\nabla_x \cdot f$ is the divergence of a vector field $f=(f_1,\dots,f_d)$ given by
\begin{align*}
	\nabla_x \cdot f = \sum_{i=1}^d \frac{\partial f_i}{\partial x_i}
\end{align*}
and $\nabla_x g$ is the gradient of a function $g$ given by
\begin{align*}
	\nabla_x g = \left( \frac{\partial g}{\partial x_1}, \dots, \frac{\partial g}{\partial x_d} \right)^{\top}.
\end{align*}
Note that \eqref{WscorePDE} is often called the continuity equation in the dynamic formulation of the optimal transport problem \cite{benamou,villani2}.
In this context, the Wasserstein score function $\Phi_i(x;\theta)$ can be viewed as a solution of the Hamilton--Jacobi equation (up to additive constant), where $\theta_i$ is adopted as the time variable.

The Wasserstein information matrix $G_W(\theta) \in \mathbb{R}^{d\times d}$ is defined as
\begin{align*}
	G_W(\theta)_{ij} &={\rm E}_{\theta} \left[ (\nabla_x \Phi_i(x;\theta))^{\top} (\nabla_x \Phi_j(x;\theta)) \right], \quad i,j=1,\dots,d.
\end{align*}
Recall that the $L^2$-Wasserstein distance $W_2(p,q)$ between two probability densities $p$ and $q$ on $\mathbb{R}^d$ is defined by
\begin{align*}
	W_2 (p,q) = \inf_{X,Y} \ {\rm E} [ \| X-Y\|^2 ] ^{\frac{1}{2}},
\end{align*}
where the infimum is taken over all joint distributions (coupling) of $(X,Y)$ with marginal distributions of $X$ and $Y$ equal to $p$ and $q$, respectively.
The Wasserstein information matrix appears in the quadratic approximation of the $L^2$ Wasserstein distance:
\begin{align*}
	W_2(p_\theta,p_{\theta+\Delta\theta})^2 = \Delta\theta^\top G_W(\theta) \Delta\theta + o(\|\Delta\theta\|^2).
\end{align*}

For a statistic $a(x)\in\mathbb{R}^l$, its Wasserstein variance ${\rm Var}_{\theta}^{\mathrm{W}}(a(x))\in\mathbb{R}^{l\times l}$ is defined by
\begin{align*}
	{\rm Var}_{\theta}^{\mathrm{W}}(a(x))_{ij}= {\rm E}_{\theta} \left[ (\nabla_x a_i(x))^{\top} (\nabla_x a_j(x)) \right], \quad i,j=1,\dots,l.
\end{align*}
Note that the Wasserstein information matrix is the Wasserstein variance of the Wasserstein score function.

\begin{lemma}[Wasserstein--Cramer--Rao inequality \citep{WIM}]
	For a statistic $a(x)\in\mathbb{R}^l$,
	\begin{align}\label{WCR inequality}
		{\rm Var}_{\theta}^{\mathrm{W}}(a(x)) \succeq \left( \frac{\partial}{\partial \theta} {\rm E}_{\theta}[a(x)]\right)^\top G_W(\theta)^{-1}\left(\frac{\partial}{\partial \theta} {\rm E}_{\theta}[a(x)]\right),
	\end{align}
	where 
	\begin{align*}
		\frac{\partial}{\partial \theta} {\rm E}_{\theta}[a(x)] := \left(\frac{\partial}{\partial \theta} {\rm E}_{\theta}[a_j(x)]\right)_{ij} \in \mathbb{R}^{d\times l}.
	\end{align*}
	In particular, if $d=l$ and $a(x)$ is an unbiased estimator of $\theta$ (${\rm E}_{\theta}[a(x)]=\theta$), then
	\begin{align*}
		{\rm Var}_{\theta}^{\mathrm{W}}(a(x)) \succeq G_W(\theta)^{-1}.
	\end{align*}
\end{lemma}

We refer to the inequality \eqref{WCR inequality} as the Wasserstein--Cramer--Rao inequality in the following.
Recently, \cite{affine} discussed its connection to robustness of an estimator against additive noise.
We also note that the Wasserstein--Cramer--Rao inequality has been obtained independently in statistical physics and called the short-time thermodynamic uncertainty relation \cite{dechant2022geometric,dechant2022geometric2,ito,otsubo2020estimating}.

In this paper, we say that an estimator is (asymptotically) Wasserstein efficient if it attains the Wasserstein--Cramer--Rao lower bound (asymptotically).
We investigate conditions of (asymptotic) Wasserstein efficiency in the following.

\section{Wasserstein efficiency in one-parameter models}\label{one-parameter}
In this section, we focus on scalar estimators for one-parameter models  $p(x; \theta)$ on $\mathbb{R}^d$ (i.e., $l=p=1$).
In this setting, attainment of the (original) Cramer--Rao lower bound has been studied well \cite{Fisher_efficiency,Fisher_efficiency2}.
Namely, a scalar estimator $a(x)$ attains the Cramer--Rao lower bound for every $\theta$ if and only if the model is a one-parameter exponential family 
\begin{align}\label{exp_family}
	p(x; \theta) = g(x) \exp (\theta T(x) - \psi(\theta))
\end{align}
and the estimator $a(x)$ is the maximum likelihood estimator of its expectation parameter $T(x)$ (or its affine transform).
Note that a one-parameter exponential family corresponds to an e-geodesic with respect to the Fisher metric in information geometry \cite{amari2000methods}.

Now, we consider the Wasserstein case.
We write $\Phi(x;\theta)=\Phi_1(x;\theta)$ for convenience.
Since the Wasserstein--Cramer--Rao inequality \eqref{WCR inequality} is derived from the Cauchy--Schwarz inequality \cite{WIM,affine}, its equality condition is obtained as follows.

\begin{theorem}\label{th_equality}
	Let $p(x; \theta)$ be a one-parameter model on $\mathbb{R}^d$ and $a(x)$ be a scalar estimator.
	Then,
	\begin{align*}
		{\rm Var}_{\theta}^{\mathrm{W}}(a(x)) \geq \frac{1}{ G_W(\theta)} \left( \frac{\partial}{\partial \theta} {\rm E}_{\theta}[a(x)]\right)^2
	\end{align*}
	and the equality holds if and only if 
	\begin{align*}
		a(x) = u(\theta)\Phi(x; \theta)  + v(\theta)
	\end{align*}
	for some $u(\theta)$ and $v(\theta)$.
\end{theorem}
\begin{proof}
	For random vectors $U$ and $V$, 
	\begin{align*}
		0 \leq {\rm E} \| tU+V \|^2 = {\rm E} [\| U \|^2] t^2 + 2 {\rm E} [U^{\top} V] t + {\rm E} [\| V \|^2]
	\end{align*}
	for every $t$.
	Thus, by considering the discriminant of the quadratic equation,
	\begin{align}
		{\rm E} [U^{\top} V]^2 \leq {\rm E} [\| U \|^2] {\rm E} [\| V \|^2], \label{CS}
	\end{align}
	where the equality holds if and only if $U$ and $V$ are linearly dependent.
	
	From the definition of the Wasserstein score function,
	\begin{align*}
		\frac{\partial}{\partial \theta_i} {\rm E}_{\theta} [a(x)] &= \int a(x) \frac{\partial}{\partial \theta_i} p(x; \theta) {\rm d} x \\
		&= -\int a(x) \nabla_{x} \cdot (p(x; \theta) \nabla_{x} \Phi(x; \theta)) {\rm d} x \\
		&= -\int (\nabla_{x} \cdot (a(x) p(x; \theta) \nabla_{{x}} \Phi(x; \theta)) - (\nabla_{{x}} a(x))^{\top} (\nabla_{{x}} \Phi(x; \theta)) p(x; \theta) ) {\rm d} x \\
		&= {\rm E}_{\theta} [(\nabla_{{x}} a(x))^{\top} (\nabla_{{x}} \Phi(x; \theta))],
	\end{align*}
	where we used the Gauss's divergence theorem and $p(x; \theta) \to 0$ as $\| x \| \to \infty$ in the fourth equality.
	Also,
	\begin{align*}
		{\rm E}_{\theta} [ \| \nabla_x a(x) \|^2 ] = {\rm Var}_{\theta}^{\mathrm{W}} [a(x)] , \quad {\rm E}_{\theta} [ \| \nabla_x \Phi(x; \theta) \|^2 ] = G_W (\theta).
	\end{align*}
	Thus, the inequality \eqref{CS} with $U =\nabla_{x} a(x)$ and $V = \nabla_{x} \Phi(x;\theta)$ is equivalent to the Wasserstien--Cramer--Rao inequality \eqref{WCR inequality}:
	\begin{align*}
		{\rm Var}_{\theta}^{\mathrm{W}}(a(x)) \geq \frac{1}{ G_W(\theta)} \left( \frac{\partial}{\partial \theta} {\rm E}_{\theta}[a(x)]\right)^2.
	\end{align*}
	Therefore, the Wasserstein--Cramer--Rao lower bound is attained if and only if $\nabla_x a(x)$ and $\nabla_x \Phi(x; \theta)$ are linearly dependent.
	This condition is rewritten as $a(x) = u(\theta) \Phi(x; \theta) + v(\theta)$ for some $u(\theta)$ and $v(\theta)$.
\end{proof}

Recently, \cite{ay2024information} provided a framework of Wasserstein information geometry by introducing the e-connection as the dual of the m-connection with respect to the Otto metric, which is defined as the Riemannian metric on the Wasserstein space \citep{Otto}.
This e-connection is different from the one in the usual information geometry, which is the dual of the m-connection with respect to the Fisher metric \cite{amari2000methods}.
The e-geodesics in the usual information geometry are given by one-parameter exponential families \eqref{exp_family}, and their Fisher score functions do not depend on $\theta$ (up to additive constant):
\begin{align*}
	\frac{\partial}{\partial \theta} \log p(x;\theta) = T(x)-\psi'(\theta).
\end{align*}
Analogously, the e-geodesics in the Wasserstein information geometry of \cite{ay2024information} are characterized as one-parameter models with fixed Wasserstein score functions (up to additive constant):
\begin{align*}
	\Phi(x; \theta) = T(x) - c(\theta).
\end{align*}
Note that it is different from the displacement interpolation \cite{villani2}, which corresponds to the geodesic with respect to the Levi-Civita connection for the Otto metric.
From this viewpoint, Theorem~\ref{th_equality} can be rewritten as follows.

\begin{corollary}\label{cor_geodesic}
	For a regular one-parameter model $p(x; \theta)$ on $\mathbb{R}^d$, a non-constant scalar estimator $a(x)$ attains the Wasserstein--Cramer--Rao lower bound for every $\theta$ if and only if the model corresponds to an e-geodesic with respect to the Otto metric (up to monotone reparametrization) and $a(x)$ is an affine function of its Wasserstein score function.
\end{corollary}
\begin{proof}
	For a parameter transformation $\widetilde{\theta}=h(\theta)$, we have
	\begin{align*}
		\frac{\partial}{\partial \widetilde{\theta}} p(x;\widetilde{\theta}) = \frac{1}{h'(\theta)} \frac{\partial}{\partial \theta} p(x;\theta).
	\end{align*}
	Thus, by setting
	\begin{align}\label{Wscore_trans}
		\Phi(x;\widetilde{\theta}) = \frac{1}{h'(\theta)} \Phi(x;{\theta}),
	\end{align}
	we obtain the continuity equation \eqref{WscorePDE} under $\widetilde{\theta}$:
	\begin{align*}
		\frac{\partial}{\partial \widetilde{\theta}} p(x;\widetilde{\theta}) + \nabla_x \cdot \left( p(x;\widetilde{\theta}) \nabla_x \Phi(x;\widetilde{\theta}) \right) = 0
	\end{align*}
	Namely, \eqref{Wscore_trans} gives the transformation rule of the Wasserstein score fuctions for reparametrization $\widetilde{\theta}=h(\theta)$.
	
	Now, from Theorem~\ref{th_equality}, an estimator $a(x)$ attains the Wasserstein--Cramer--Rao lower bound for every $\theta$ if and only if $a(x) = u(\theta)\Phi(x; \theta)  + v(\theta)$ for every $\theta$.
	Since $a(x)$ is not constant, we have $u(\theta) \neq 0$.
	Also, from the regularity of $p(x; \theta)$, $u(\theta)$ is continuous.
	Thus, $u(\theta)$ does not change sign.
	Therefore, the function
	\begin{align*}
		h(\theta) = \int_0^{\theta} \frac{1}{u(\theta')} {\rm d} \theta',
	\end{align*}
	is monotone.
	Consider the parameter transformation $\widetilde{\theta}=h(\theta)$.
	From \eqref{Wscore_trans}, the Wasserstein score function under $\widetilde{\theta}$ is
	\begin{align*}
		\Phi(x;\widetilde{\theta}) = \frac{1}{h'(\theta)} \Phi(x;{\theta}) = u(\theta) \Phi(x;{\theta}) = a(x)-v(\theta),
	\end{align*}
	which does not depend on $\widetilde{\theta}$ up to additive constant.
	Therefore, the model $p(x;\widetilde{\theta})$ is an e-geodesic with respect to the Otto metric as introduced in \cite{ay2024information}.
	In other words, the model $p(x;\theta)$ is an e-geodesic with respect to the Otto metric up to monotone reparametrization.
\end{proof}

Since an e-geodesic with respect to the Otto metric can be viewed as a Wasserstein analogue of the one-parameter exponential family, Corollary~\ref{cor_geodesic} is a natural generalization of the classical result on attainment of the Cramer--Rao lower bound \cite{Fisher_efficiency,Fisher_efficiency2}.
The parameter transformation $\widetilde{\theta}=h(\theta)$ in the proof is similar to a unit-speed parametrization of a curve on a Riemannian manifold.
We give several examples of Wasserstein efficient estimators for $d=1$.

\begin{proposition}
	\begin{enumerate}
		\item The estimator $a(x) = x$ is Wasserstein efficient if and only if the model is the location family 
		\begin{align}
			p(x; \theta) = f(x-\theta). \label{loc}
		\end{align}
		
		\item The estimator $a(x) = x^2$ is Wasserstein efficient if and only if the model is the scale family
		\begin{align}
			p(x; \theta) = \frac{1}{\theta} f \left( \frac{x}{\theta} \right). \label{scale}
		\end{align}
	\end{enumerate}
\end{proposition}
\begin{proof}
	\begin{enumerate}
		\item The Wasserstein score function of the location family \eqref{loc} is
		\begin{align*}
			\Phi(x; \theta) = x-\theta,
		\end{align*}
		where we assume ${\rm E}_{\theta}[x] = \theta$ without loss of generality.
		Thus, we obtain the result by using Theorem~\ref{th_equality} with $u(\theta)=1$ and $v(\theta)=\theta$.
		
		\item The Wasserstein score function of the scale family \eqref{scale} is
		\begin{align*}
			\Phi(x; \theta) = \frac{x^2}{2 \theta} - \frac{\theta}{2},
		\end{align*}
		where we assume ${\rm E}_{\theta}[x^2] = \theta^2$ without loss of generality.
		Thus, we obtain the result by using Theorem~\ref{th_equality} with $u(\theta)=\theta$ and $v(\theta)=0$.
	\end{enumerate}
\end{proof}

\section{Wasserstein efficiency in location-scale families}\label{location-scale}
In this section, we consider location-scale families on $\mathbb{R}$:
\begin{align}\label{lcf}
	p(x; \theta) = \frac{1}{\sigma}f\left(\frac{x-\mu}{\sigma}\right), \quad \theta=(\mu,\sigma),
\end{align}
where $f$ is a probability density on $\mathbb{R}$ with mean zero and variance one (e.g., ${\rm N}(0,1)$). 
The mean and variance of $p(x; \theta)$ are $\mu$ and $\sigma^2$, respectively.
Its Wasserstein score function is 
\begin{align}\label{Wscore}
	\Phi_{\mu}(x;\theta) = x-\mu, \quad \Phi_{\sigma}(x;\theta) = \frac{(x-\mu)^2}{2\sigma} - \frac{\sigma}{2},
\end{align}
which can be confirmed by substitution into \eqref{WscorePDE}:
\begin{align*}
	\frac{\partial}{\partial \mu} p(x;\theta) + \frac{\partial}{\partial x} \left( p(x;\theta) \frac{\partial}{\partial x} (x-\mu) \right) = \frac{\partial}{\partial \mu} p(x;\theta) + \frac{\partial}{\partial x} p(x;\theta) = 0,
\end{align*}
\begin{align*}
	& \frac{\partial}{\partial \sigma} p(x;\theta) + \frac{\partial}{\partial x} \left( p(x;\theta) \frac{\partial}{\partial x} \left( \frac{(x-\mu)^2}{2\sigma} - \frac{\sigma}{2} \right) \right) \\
	= & \left( -\frac{1}{\sigma^2} f\left(\frac{x-\mu}{\sigma}\right) + \frac{1}{\sigma} f'\left(\frac{x-\mu}{\sigma}\right) \left(-\frac{x-\mu}{\sigma^2}\right) \right) + \frac{\partial}{\partial x} \left( \frac{1}{\sigma} f \left(\frac{x-\mu}{\sigma}\right) \cdot \frac{x-\mu}{\sigma} \right) \\
	=& 0.
\end{align*}

Suppose that we have $n$ independent observations $x_1,\dots,x_n$ from $p(x;\theta)$.	
Then, the Wasserstein estimator $\hat{\theta}_W=(\hat{\mu}_W,\hat{\sigma}_W)$ is defined as the zero point of the Wasserstein score function \citep{WIM}:
\begin{align*}
	\sum_{t=1}^n \Phi_{\mu}(x_t; \hat{\theta}_W) = \sum_{t=1}^n \Phi_{\sigma}(x_t; \hat{\theta}_W) = 0.
\end{align*}	
From \eqref{Wscore}, it is given by the sample mean and sample standard deviation: 
\begin{align}\label{West}
	\hat{\mu}_W = \bar{x} = \frac{1}{n}\sum_{t=1}^n x_t,\quad  \hat{\sigma}_W = \sqrt{\frac{1}{n}\sum_{t=1}^n (x_t-\bar{x})^2}.
\end{align}

\begin{theorem}\label{th: asymptotic efficiency}
	For the location-scale family \eqref{lcf}, the Wasserstein estimator $\hat{\theta}_W=(\hat{\mu}_W,\hat{\sigma}_W)$ in \eqref{West} asymptotically attains the Wasserstein--Cramer--Rao lower bound:
	\begin{align*}
		n \left( {\rm Var}_{\theta}^{\mathrm{W}}(\hat{\theta}_W) - \frac{1}{n} \left(\frac{\partial}{\partial \theta} {\rm E}_{\theta} [\hat{\theta}_W] \right)^\top G_W(\theta)^{-1} \left(\frac{\partial}{\partial \theta} {\rm E}_{\theta} [\hat{\theta}_W]\right) \right) \to \begin{pmatrix} 0 & 0 \\ 0 & 0 \end{pmatrix}
	\end{align*}
	as $n\to\infty$.
\end{theorem}
\begin{proof}
	From \eqref{West},
	\begin{align*}
		\nabla \hat{\mu}_W = \frac{1}{n} \left( 1,\dots,1 \right),\quad \nabla \hat{\sigma}_W = \frac{1}{n\hat{\sigma}_W} ( x_1-\bar{x},\dots, x_n-\bar{x} ).
	\end{align*}
	Thus,
	\begin{align*}
		{\rm Var}_{\theta}^{\mathrm{W}} (\hat{\theta}_W)= \begin{pmatrix} {\rm E}_{\theta} [\nabla \hat{\mu}_W \cdot \nabla \hat{\mu}_W] & {\rm E}_{\theta} [\nabla \hat{\mu}_W \cdot \nabla \hat{\sigma}_W] \\ {\rm E}_{\theta} [\nabla \hat{\sigma}_W \cdot \nabla \hat{\mu}_W] & {\rm E}_{\theta} [\nabla \hat{\sigma}_W \cdot \nabla \hat{\sigma}_W] \end{pmatrix} = \frac{1}{n}\begin{pmatrix}
			1 & 0 \\
			0 & 1
		\end{pmatrix}.
	\end{align*}
	On the other hand, 
	\begin{align*}
		G_W(\theta)= n \begin{pmatrix}
			1 & 0 \\
			0 & 1
		\end{pmatrix}, \quad {\rm E}_{\theta}[\hat{\theta}_W]= \begin{pmatrix} \mu \\ c_n \sigma \end{pmatrix},
	\end{align*}
	where $c_n$ is the expected value of the sample standard deviation of $x_1,\dots,x_n \sim f$.
	Thus,
	\begin{align*}
		\left(\frac{\partial}{\partial \theta} {\rm E}_{\theta}[\hat{\theta}_W]\right)^\top G_W(\theta)^{-1} \left(\frac{\partial}{\partial \theta} {\rm E}_{\theta}[\hat{\theta}_W]\right) = \frac{1}{n}\begin{pmatrix}
			1 & 0 \\
			0 & c_n^2
		\end{pmatrix}.
	\end{align*}
	From the law of large numbers and continuous mapping theorem, 
	\begin{align*}
		\hat{\sigma}_W^2 = \frac{1}{n} \sum_{t=1}^n x_t^2 - \left( \frac{1}{n} \sum_{i=1}^n x_t \right)^2 \overset{p}{\to} (\mu^2+\sigma^2) - \mu^2 = \sigma^2
	\end{align*}
	as $n \to \infty$.
	Then, from continuous mapping theorem, $\hat{\sigma}_W \overset{p}{\to} \sigma$ as $n\to\infty$.
	Also, the Markov inequality
	\begin{align*}
		{\rm E}[\hat{\sigma}_W 1 (\hat{\sigma}_W>M)] \leq \frac{{\rm E}[\hat{\sigma}_W^2]}{M}=\frac{\sigma^2}{M},
	\end{align*}
	shows the uniform integrability of $\hat{\sigma}_W$: $\sup_n {\rm E}[\hat{\sigma}_W 1 (\hat{\sigma}_W>M)] \to 0$ as $M \to \infty$.
	Therefore, we have ${\rm E}[\hat{\sigma}_W] \to \sigma$ and thus $c_n \to 1$ as $n \to \infty$.
	Hence,
	\begin{align*}
		n \left( {\rm Var}_{\theta}^{\mathrm{W}}(\hat{\theta}_W) -\left(\frac{\partial}{\partial \theta} {\rm E}_{\theta} [\hat{\theta}_W] \right)^\top G_W(\theta)^{-1} \left(\frac{\partial}{\partial \theta} {\rm E}_{\theta} [\hat{\theta}_W]\right) \right) = \begin{pmatrix}
			0 & 0 \\
			0 & 1-c_n^2
		\end{pmatrix} \to O
	\end{align*}
	as $n\to\infty$.
\end{proof}

Since the Wasserstein variance of an estimator quantifies its robustness in terms of the increase of its variance due to noise contamination \cite{affine}, Theorem~\ref{th: asymptotic efficiency} implies that the Wasserstein estimator is robust against additive noise in location-scale families.
We confirm this for the Laplace distribution:
\begin{align}\label{laplace}
	p(x;\theta) = \frac{1}{\sqrt{2}\sigma}\exp\left(-\sqrt{2}\left|\frac{x-\mu}{\sigma}\right|\right).
\end{align}
For $\mu$, the Wasserstein estimator $\hat{\mu}_{\mathrm{W}}$ is the sample mean while the MLE $\hat{\mu}_{\mathrm{ML}}$ is the sample median.
Both estimators are unbiased and their variances are 
\begin{align*}
	\mathrm{Var}_{\theta}(\hat{\mu}_{W}) = \frac{\sigma^2}{n}, \quad \mathrm{Var}_{\theta}(\hat{\mu}_{\mathrm{ML}}) = \frac{\sigma^2}{2n}+o \left( \frac{1}{n} \right)
\end{align*}
as $n \to \infty$.
Now, suppose that we have noisy observations $\tilde{x}_1=x_1+z_1,\dots,\tilde{x}_n=x_n+z_n$ instead of $x_1,\dots,x_n$, where $z_1,\dots,z_n \sim {\rm N}(0,\varepsilon^2)$ are independent Gaussian noise with variance $\varepsilon^2$.	
Then, 
\begin{align*}
	\frac{\mathrm{Var}_{\theta}(\hat{\mu}_{\mathrm{W}}(x_1+z_1,\dots,x_n+z_n))-\mathrm{Var}_{\theta}(\hat{\mu}_{\mathrm{W}}(x_1,\dots,x_n))}{\varepsilon^2} = \frac{1}{n},
\end{align*}
which does not depend on $\varepsilon^2$.
On the other hand, as shown in Appendix,
\begin{align}\label{laplace_W}
	\frac{\mathrm{Var}_{\theta}(\hat{\mu}_{\mathrm{ML}}(x_1+z_1,\dots,x_n+z_n))-\mathrm{Var}_{\theta}(\hat{\mu}_{\mathrm{ML}}(x_1,\dots,x_n))}{\varepsilon^2} \approx \frac{2 \sigma}{\sqrt{\pi} n \varepsilon}
\end{align} 
for large $n$, which diverges as $\varepsilon^2 \to 0$.
Therefore, the Wasserstein estimator is more robust than MLE against small noise.
It is an interesting future problem to investigate the Wasserstein efficiency in comparison to Fisher efficiency for models other than location-scale families.

\begin{appendix}	
	\section{Derivation of \eqref{laplace_W}}
	From Section 13 of \cite{ferguson2017course}, the asymptotic distribution of the sample median of $n$ independent samples $x_1,\dots,x_n \sim p$ is
	\begin{align*}
		\sqrt{n} ({\rm median}(x_1,\dots,x_n)-m) \to {\rm N} \left(0,\frac{1}{4 p(m)^2} \right),
	\end{align*}
	as $n \to \infty$, where $m$ is the median of $p$.
	Thus, 
	\begin{align}\label{median_var}
		{\rm Var} ({\rm median}(x_1,\dots,x_n)) \approx \frac{1}{4 n p(m)^2} 
	\end{align}
	for large $n$.
	Therefore, for the Laplace distribution \eqref{laplace},
	\begin{align}\label{laplace_var}
		\mathrm{Var}_{\t}(\hat{\mu}_{\mathrm{ML}}(x_1,\dots,x_n)) \approx \frac{\sigma^2}{2n}
	\end{align}
	for large $n$.
	On the other hand, the probability density of the noisy observation $\tilde{x}=x+z$ with $z \sim {\rm N}(0,\varepsilon^2)$ is given by the convolution
	\begin{align*}
		p(\tilde{x}) &= \int_{-\infty}^\infty \frac{1}{\sqrt{2} \sigma} \exp \left( -\sqrt{2} \left| \frac{\tilde{x}-z-\mu}{\sigma} \right| \right) \frac{1}{\sqrt{2 \pi \varepsilon^2}} \exp \left( -\frac{z^2}{2 \varepsilon^2} \right) {\rm d} z.
	\end{align*}
	At the median $\tilde{x}=\mu$,
	\begin{align*}
		p(\tilde{x}=\mu) &= \int_{-\infty}^\infty \frac{1}{\sqrt{2} \sigma} \exp \left( -\sqrt{2} \left| \frac{-z}{\sigma} \right| \right) \frac{1}{\sqrt{2 \pi \varepsilon^2}} \exp \left( -\frac{z^2}{2 \varepsilon^2} \right) {\rm d} z \\
		&= 2 \int_{0}^\infty \frac{1}{2 \sqrt{\pi} \sigma \varepsilon} \exp \left( -\sqrt{2} \frac{z}{\sigma} -\frac{z^2}{2 \varepsilon^2} \right) {\rm d} z \\
		&= \frac{1}{\sqrt{\pi} \sigma \varepsilon}  \int_{0}^\infty \exp \left( -\frac{1}{2 \varepsilon^2} \left( z+\frac{\sqrt{2} \varepsilon^2}{\sigma} \right)^2 + \frac{\varepsilon^2}{\sigma^2 } \right) {\rm d} z \\
		&= \frac{\sqrt{2}}{\sigma} \Phi \left( -\frac{\sqrt{2} \varepsilon}{\sigma} \right) \exp \left( \frac{\varepsilon^2}{\sigma^2 } \right) \\
		&= \frac{1}{\sqrt{2} \sigma} - \frac{\sqrt{2}}{\sqrt{\pi} \sigma^2} \varepsilon + O(\varepsilon^2)
	\end{align*}
	as $\varepsilon \to 0$, where $\Phi$ is the cumulative distribution function of the standard Gaussian ${\rm N}(0,1)$.
	Thus, by using \eqref{median_var},
	\begin{align}\label{conv_var}
		{\rm Var}_{\theta}(\hat{\mu}_{\mathrm{ML}}(x_1+z_1,\dots,x_n+z_n)) \approx \frac{\sigma^2}{2n} \left( 1+\frac{4}{\sqrt{\pi} \sigma} \varepsilon + O(\varepsilon^2) \right)
	\end{align}
	for large $n$.
	Combining \eqref{laplace_var} and \eqref{conv_var} yields \eqref{laplace_W}.	
\end{appendix}

	\section*{Acknowledgments}
	We are grateful to the referees for constructive comments.
We thank Keiya Sakabe and Sosuke Ito for helpful comments.
We thank Frank Nielsen for pointing out a typo in an earlier version of this manuscript.
Takeru Matsuda was supported by JSPS KAKENHI Grant Numbers 19K20220, 21H05205, 22K17865 and JST Moonshot Grant Number JPMJMS2024.

\end{document}